\documentclass[11pt,a4paper]{article}

\usepackage{epsf,epsfig,amsfonts,amsgen,amsmath,amstext,amsbsy,amsopn,amsthm
}
\usepackage{amsmath,times,mathptmx}
\usepackage{enumitem}
\usepackage{amsfonts,amsthm,amssymb}
\usepackage{amsfonts}
\usepackage{graphics}
\usepackage{latexsym,bm}
\usepackage{amsfonts,amsthm,amssymb,bbding}
\usepackage{indentfirst}
\usepackage{graphicx}
\usepackage{color}
\usepackage[colorlinks=true,anchorcolor=blue,filecolor=blue,linkcolor=blue,urlcolor=blue,citecolor=blue]{hyperref}
\usepackage{float}
\usepackage{tikz}

\DeclareMathAlphabet{\mathcal}{OMS}{cmsy}{m}{n}
\DeclareSymbolFont{largesymbols}{OMX}{cmex}{m}{n}
\allowdisplaybreaks[4]
\usepackage{moresize}

\newtheorem{thm}{Theorem}

\newtheorem{lemma}{Lemma}
\newtheorem{false statement}{False statement}

\theoremstyle{definition}

\newtheorem{claim}{Claim}

\newtheorem{remark}
{Remark}

\newtheorem{case}{Case}
\newtheorem{casee}{Case}

\newtheorem{problem}{Problem}

\usepackage{needspace}

\hypersetup{pdftitle={Spectral extremal problems on 1-planar graphs without Friendship graph},
  pdfauthor={Jiamin Li, Dan Li},
  pdfsubject={Friendship-free 1-planar spectral extremal graphs}}
  
  \usetikzlibrary{
  	calc,
  	intersections,
  	positioning
  }

\begin{document}

\title{Spectral extremal problems on 1-planar graphs without Friendship graph
\footnote{Supported by Tianshan Talent Training Program (No. 2024TSYCCX0013),
	Natural Science Foundation of Xinjiang Uygur Autonomous Region (No. 2024D01C41),
	 The Basic scientific research in universities of Xinjiang
	Uygur Autonomous Region (XJEDU2025P001) and NSFC (No. 12361071).}}
\author{{Jiamin Li, Dan Li\thanks{Corresponding author. E-mail: ldxjedu@163.com.},  Yuanyuan Chen}\\
{\footnotesize College of Mathematics and Systems Science, Xinjiang University, Urumqi 830046, China}}
\date{}

\maketitle {\flushleft\large\bf Abstract:}
Let $\textit{spex}_{\mathcal{P}_1}(n,F)$ be the maximum spectral radius among all
$n$-vertex $F$-free $1$-planar graphs.  Define
$F_t$ as the friendship graph formed by $t$ triangles sharing
exactly one common vertex.
Tait and Tobin (2017)~\cite{Tait2017} used the fundamental structure of
spectral extremal graphs to determine the unique planar graph with maximum
spectral radius for sufficiently large order. Subsequently, Zhang, Wang and
Wang (2024)~\cite{Zhang2024} characterized the corresponding extremal graph
in the class of $1$-planar graphs.
In this paper, we focus on $F_t$-free $1$-planar graphs and establish a
structural theorem for their spectral extremal graphs for all $t\geq1$ and
sufficiently large $n$. More precisely, every extremal graph is connected
and contains a copy of $K_{2,n-2}$, and for $t\geq2$ the two distinguished
vertices are adjacent and the subgraph induced by the remaining vertices is a
bipartite graph. Based on this structure result together with the drawing properties of $K_{3,6}$,
we determine $\textit{spex}_{\mathcal{P}_1}(n,F_t)$ and characterize its
unique extremal graph.

\vspace{0.1cm}
\begin{flushleft}
\end{flushleft}

\section{Introduction}\label{sec:introduction}

All graphs considered in this paper are finite, simple and undirected.
For a graph $G$, let $A(G)$ be its adjacency matrix, $\rho(G)$ be
its spectral radius (i.e., the maximum modulus of eigenvalues of $A(G)$).
Given a graph family $\mathcal{F}$, a graph is said to be
$\mathcal{F}$-free if it does not contain any $F\in\mathcal{F}$ as a
subgraph. When $\mathcal{F}=\{F\}$, we write $F$-free instead of
$\mathcal{F}$-free. Let $\operatorname{spex}_{\mathcal{P}_1}(n,F)$ be the maximum spectral radius over all $F$-free $1$-planar graphs of order $n$, and $\operatorname{SPEX}_{\mathcal{P}_1}(n,F)$ be the family of $F$-free $1$-planar graphs of order $n$ with spectral radius equal to $\operatorname{spex}_{\mathcal{P}_1}(n,F)$.
For two vertex-disjoint graphs $G_1$ and $G_2$, let $G_1\cup G_2$ be the disjoint union of them, and $G_1\vee G_2$ denote their $join$ which is the graph obtained by adding all possible edges between $G_1$ and $G_2$. A Perron vector of $G$ is a nonzero nonnegative eigenvector of $A(G)$
corresponding to $\rho(G)$, we write $X=(x_v)_{v\in V(G)}^T$ and normalize
it by $\max_{v\in V(G)}x_v=1$. If $G$ is connected, then $X>0$ by the
Perron--Frobenius theorem.

As usual, $V(G)$ and $E(G)$ denote the vertex set and edge set of $G$.
For $U\subseteq V(G)$, $G[U]$ denotes the subgraph of $G$ induced by $U$,
and for $v\in V(G)$, $N_G(v)$ denotes the set of neighbors of $v$ in $G$.
Write $e(G)=|E(G)|$, $d_G(v)=|N_G(v)|$,
$\Delta(G)=\max_{v\in V(G)}d_G(v)$, $e(U)=|E(G[U])|$,
$d_U(v)=|N_G(v)\cap U|$ and $N_G(U)=\bigcup_{u\in U}N_G(u)$.
For two disjoint subsets $U$ and $V$ of $V(G)$, let \(E_G[U,V]\) be the set of edges with one end vertex in \(U\) and the other in \(V\), and \(G[U,V]\) be the bipartite graph with vertex set \(U\cup V\) and edge set \(E_G[U,V]\).  
A vertex subset $U$ is independent if $e(U)=0$. A pair $(A,D)$ is a
bipartition  of $V(G)$ if $V(G)=A\cup D, A\cap D=\varnothing$, and both $A$ and $D$ are independent set. 

A graph is called planar if it can be drawn in the plane without crossed edges. Research on the spectral radius of planar graphs goes back to Schwenk
and Wilson~\cite{Schwenk1978}.
Boots and Royle~\cite{Boots1991}, and independently Cao and
Vince~\cite{Cao1993}, proposed a conjecture on the spectral extremal
planar graph, which Tait and Tobin~\cite{Tait2017} confirmed for
sufficiently large $n$ by determining its fundamental structure $K_2\vee P_{n-2}$. A natural generalization reads:
given a graph $F$, what is
the maximum spectral radius of an $F$-free planar graph of order $n$?
This problem has been studied for various forbidden configurations,
such as $k$ edge-disjoint cycles~\cite{ZhaiLiu2024},
vertex-disjoint cycles $tC_l$~\cite{Fang2024},
wheels $W_k$, friendship graphs $F_t$ and
matchings $(t+1)K_2$~\cite{WangHuangLin2024},
intersecting cycles $B_{t,l}$~\cite{YinLi2026}
and $C_{k,l}$~\cite{LiLiYinChen2026},
linear forests $tP_l$~\cite{YinLiMeng2025}, where $B_{t,l}$ denotes the graph obtained by  $t$ edge-disjoint $l$-cycles sharing a common vertex,  and $C_{k,l}$  as a graph obtained from $C_k \cup C_l$ such that the two cycles share a common vertex for  $l \ge k \ge 3$.
These studies provide spectral extremal results for sufficiently
large $n$ under the respective parameter conditions.

Inspired by these results, we turn our attention to $1$-planar graphs.
A graph is $1$-planar if it has a drawing in the plane in which each edge
is crossed at most once. In 2024, Zhang, Wang and Wang~\cite{Zhang2024} proved that
$\rho(G)\leq5+\sqrt{2n+5}$ for every $n$-vertex
$1$-planar graph with $n\geq7$, and determined the unique spectral
extremal graph in $1$-planar graph when $n\geq6.23\times10^{18}$.
Xu and Chang~\cite{Xu2025} determined
$\operatorname{SPEX}_{\mathcal{P}_1}(n,K_r)$ for $r\in\{3,4\}$ and
restricted $\operatorname{SPEX}_{\mathcal{P}_1}(n,K_5)$ to an explicitly
described family for sufficiently large $n$.
More recently, Li, Wang and Zhao~\cite{LiWangZhao2026} determined
$\operatorname{SPEX}_{\mathcal{P}_1}(n,F)$ for
$F\in\{K_5,C_5,2C_5\}$ with sufficiently large $n$, and established
a structural theorem for connected spectral extremal $F$-free
$1$-planar graphs under suitable conditions on $F$. Naturally, we focus on the following problem.

\begin{problem}\label{prob:friendship}
	For $t\geq1$ and sufficiently large $n$, which graphs belong to
	$\operatorname{SPEX}_{\mathcal{P}_1}(n,F_t)$, and what is their
	fundamental structure?
\end{problem}

In this paper, we completely resolve Problem \ref{prob:friendship}. We first establish a structural theorem for graphs in
$\operatorname{SPEX}_{\mathcal{P}_1}(n,F_t)$.

\begin{thm}\label{thm1}
Let $t\geq1$ and $n\geq \max \{1.23 \times 10^{18}, 2t+1,  
 \frac{2t}{\alpha} ,  \frac{f(t)^2-f(t)+4}{2}, \frac{g(t)^2-g(t)+4}{2}, 7t-$ $12 \}$ be integers. Suppose
$G \in \operatorname{SPEX}_{\mathcal{P}_1}(n,F_t)$, let $X=(x_v)_{v\in V(G)}^T$ be its positive eigenvector corresponding to $\rho=\rho(G)$ with $\max_{v\in V(G)}x_v=1$. Then $G$ is connected and there exist two
distinct vertices $u',u''\in V(G)$ such that
$R=N_G(u')\cap N_G(u'')=V(G)\setminus\{u',u''\}$ and
$x_{u'}=x_{u''}=1$.
Moreover, the following
assertions hold.
\begin{enumerate}[label=(\roman*)]
\item If $t=1$, then $G\cong K_{2,n-2}$.
\item If $t\geq2$, then  $G\cong K_2\vee H$, where $H=G[R]$
satisfies $e(H)=6m(H)=6(t-2)$ and $H$ has a bipartition $(A,D)$ with $|A|=t-2$ ,
$d_H(a)=6$ for every $a\in A$.
\end{enumerate}
\end{thm}

\vspace*{2mm}
Combining this structural theorem with matching theory as well as the drawing properties
of $K_{3,6}$ derived from Kleitman's parity result~\cite{Kleitman1976},
we determine $\textit{spex}_{\mathcal{P}_1}(n,F_t)$ and characterize its
unique extremal graph for every $t\geq1$ and sufficiently large $n$.
This extends the study of $F_t$-free planar
graphs~\cite{WangHuangLin2024} to the $1$-planar setting.

\begin{thm}\label{thm2}
Let $t\geq1$ and $n\geq \max \{1.23 \times 10^{18}, 2t+1,  
\frac{2t}{\alpha},  \frac{f(t)^2-f(t)+4}{2}, \frac{g(t)^2-g(t)+4}{2}, 7t-$ $12 \}$ be integers, and  $G\in$
$\operatorname{SPEX}_{\mathcal{P}_1}(n,F_t)$. Then
\[
G \cong
\begin{cases}
K_{2,n-2},&\text{if }t=1,\\[3pt]
K_2\vee \bigl((t-2)S_6\cup(n-7t+12)K_1\bigr),&\text{if }t\geq2.
\end{cases}
\]
\end{thm}

For $t=1$, Xu and Chang~\cite{Xu2025} proved that the extremal $K_3$-free graph is $K_{2,n-2}$ when $n$ is sufficiently large, where the threshold is at least $6.23\times 10^{18}$. In this paper, we reduce this bound to $1.23\times 10^{18}$. For $t=2$, the extremal graph is
$K_2\vee (n-2)K_1$. For $t\geq3$, the remaining subgraph consists of $t-2$
vertex-disjoint copies of $S_6$ together with isolated vertices. This differs
from the planar extremal graph, whose remaining graph consists of one
path $P_{2t-3}$ and isolated vertices~\cite{WangHuangLin2024}.

\vspace*{2mm}
An outline of the paper is as follows.
In Section~\ref{sec:thm1}, we present several useful lemmas and prove the
Theorem~\ref{thm1}, which establishes the fundamental structure of
spectral extremal $F_t$-free $1$-planar graphs.
Section~\ref{sec:thm2} studies the drawing properties of $K_{3,6}$
and combine them with Theorem~\ref{thm1} to prove Theorem~\ref{thm2},
thereby characterizing the unique extremal graph attaining
$\textit{spex}_{\mathcal{P}_1}(n,F_t)$ for sufficiently large $n$.

\vspace*{2mm}

\section{Proof of Theorem \ref{thm1}}\label{sec:thm1}

The structural results established in Theorem~\ref{thm1} serve as the key building block for deriving Theorem~\ref{thm2}.
Before proceeding, we first list several useful lemmas. The first four provide the edge number and spectral estimates used below.

\begin{lemma}\cite{Fabrici2007}.\label{lem:edge}
If $G$ is a $1$-planar graph of order $n\geq3$, then $e(G)\leq4n-8$.
\end{lemma}

A graph is $k$-degenerate if every nonempty subgraph has a vertex of degree
at most $k$.
\begin{lemma}\cite{Fabrici2007}.\label{lm2}
Every $1$-planar graph is $7$-degenerate.
\end{lemma}

\begin{lemma}\cite{Huang2021}.\label{lem:bipartite}
Let $G$ be a bipartite $1$-planar graph with part sizes $s$ and $t$,
where $2\leq s\leq t$. Then $e(G)\leq2(s+t)+4s-12$.
\end{lemma}

\begin{lemma}\cite{Zhang2024}.\label{lem:spectral}
If $G$ is a $1$-planar graph of order $n\geq7$, then
$\rho(G)\leq5+\sqrt{2n+5}.$
\end{lemma}

\begin{lemma}\cite{Xu2025}.\label{lem:rayleigh}
Let $G$ and $G'$ be graphs on the same vertex set, and  $X$ be a nonzero nonnegative eigenvector of $A(G)$ corresponding to $\rho(G)$.
Then
\[
\rho(G')-\rho(G)\geq
\frac{X^T\bigl(A(G')-A(G)\bigr)X}{X^TX}.
\]
In particular, if the numerator is positive, then $\rho(G')>\rho(G)$.
\end{lemma}

A matching is a set of pairwise vertex-disjoint edges, and
$m(G)$ is size of a maximum matching of $G$. A matching covers its endpoints and is perfect
if it covers $V(G)$. $G$ is factor-critical if $G-v$ has a perfect
matching for every vertex $v\in V(G)$.
\begin{lemma}\cite{Lovasz1986}.\label{lem:GE}
Let $G$ be any graph. Denote by $D$ the set of all vertices in $G$ which are not covered by at least one maximum matching of $G$. Let $A$ be the set of vertices in $V(G)\setminus D$ adjacent to at least one vertex in $D$. Finally let $C=V(G)\setminus (A\cup D)$. Then
\begin{itemize}
	\item[(a)] the components of the subgraph induced by $D$ are factor-critical,
	\item[(b)] the subgraph induced by $C$ has a perfect matching.
	\item[(c)] if $M$ is any maximum matching of $G$, it contains a near-perfect matching of each component of $D$, a perfect matching of each component of $C$ and matches all points of $A$ with points in distinct components of $D$,
	
\end{itemize}
\end{lemma}

\begin{remark}\label{rk111}

		Apply Lemma~6 to a graph $H$, and retain the notation $A,C,D$.
		Let $D_1,\ldots,D_l$ be the connected components of $H[D]$.
		By Lemma~\ref{lem:GE} (a), each $D_i$ is factor-critical, so we  write
		$|V(D_i)|=2r_i+1$. By the definition of $A$, $e(C,D)=0$.
		Let $M$ be a maximum matching of $H$. By Lemma~\ref{lem:GE} (c),
		$M\cap E(D_i)$ has $r_i$ edges for every $i$,
		$M\cap E(H[C])$ has $\frac{|C|}{2}$ edges, and $M$ has exactly
		$|A|$ edges between $A$ and $D$. These account for all
		edges of $M$. Therefore,
		\[
		m(H)=|M|=|A|+\frac{|C|}{2}+\sum_{i=1}^{l}r_i.
		\]
		
	\end{remark}

\begin{lemma}\label{lem:corebounds}
Let $G$ be a $1$-planar graph on $n$ vertices with two distinct vertices $u',u''$, and 
$H=G[V(G)\setminus\{u',u''\}]$, $d_{H}(u')=d_{H}(u'')=n-2$. Let $\varepsilon=1$ if $u'u''\in E(G)$ and
$\varepsilon=0$ otherwise. Then
$\Delta(H)\leq6$ and $e(H[U])\leq2|U|-\varepsilon,$ where $\varnothing\neq U\subseteq V(H).$
\end{lemma}

\noindent \emph{Proof.}
Suppose for contradiction that there exists a vertex $v\in V(H)$ such that $d_{H}(v)\ge7$, then $G[N_H(v)\cup \{u',u'',v\}]$ contains a copy of $K_{3,7}$, contrary to
Lemma~\ref{lem:bipartite}. Thus, $\Delta(H)\leq6$.
For every  $\varnothing\neq U\subseteq V(H)$, by Lemma~\ref{lem:edge},
\[
e(H[U])+2|U|+\varepsilon
=e\bigl(G[U\cup\{u',u''\}]\bigr)
\leq4(|U|+2)-8=4|U|.
\]
Thus, $e(H[U])\leq2|U|-\varepsilon.$ 

\begin{lemma}\label{lem:core}
Let $G$ be a $1$-planar graph on $n$ vertices with two distinct vertices $u',u''$, set
$H=G[V(G)\setminus\{u',u''\}]$ and $d_{H}(u')=d_{H}(u'')=n-2$. Then $e(H)\le6m(H)$.
If  $u'u''\in E(G)$ and $e(H)=6m(H)$, then  $H$ has a
bipartition $(A,D)$ such that $|A|=m(H)$ and $d_H(a)=6$ for every
$a\in A$.
\end{lemma}

\noindent \emph{Proof.}
Apply Remark \ref{rk111}, and retain its notation $A,C,D,D_i,r_i$. Then $e(C,D)=0 $.
Let $E_A$ be the set of edges of $H$ with at least one endpoint in $A$.
By Lemma~\ref{lem:corebounds},
$|E_A|=\sum_{a\in A}d_H(a)-e(H[A])\leq6|A|$.
If $C\neq\varnothing$, then
$e(H[C])\leq2|C|<3|C|=6\frac{|C|}{2}$.
If $C=\varnothing$, then trivially $e(H[C])=0$. For $D_i$, if $r_i\geq1$ , then
$e(D_i)\leq2(2r_i+1)=4r_i+2\leq6r_i.$ If $r_i=0$ , then $e(D_i)=0$.
Summarizing these contributions, 
\[
\begin{split}
e(H)&=|E_A|+e(H[C])+\sum_i e(D_i)\leq6|A|+6\frac{|C|}{2}+6\sum_i r_i
=6m(H).
\end{split}
\]
Equality holds if and only if all inequalities in each step become equalities.
Now suppose $u'u''\in E(G)$ and $e(H)=6m(H)$. Then $|E_A|=6|A|$ and $C=\varnothing$.  If $r_i\geq1$, by
Lemma~\ref{lem:corebounds}, then $e(D_i)\leq2(2r_i+1)-1=4r_i+1<6r_i.$ Thus, $r_i=0$ for every $i$. Hence $D$ is an independent vertex set, and
\[
0=6|A|-|E_A|
=\sum_{a\in A}\bigl(6-d_H(a)\bigr)+e(H[A]).
\]
Every summand is nonnegative. Thus $A$ is an independent vertex set and $d_H(a)=6$
for every $a\in A$. Finally, by Remark \ref{rk111}, $|A|=m(H)$.

\begin{lemma}\label{lem:friendship}
Let $t\geq2$ be an integer and  $H$ be a graph with $|V(H)|\geq2t-1$. Then  $K_2\vee H\text{ is }F_t\text{-free}$ if and only if $m(H)\leq t-2.$
\end{lemma}

\noindent \emph{Proof.}
Let $V(K_2)=\{u',u''\}$. We first prove the sufficiency by contradiction. Assume that $K_2\vee H$ contains a copy
of $F_t$ with common vertex $z$. Let $I$ be the set of independent edges remaining after deleting $z$ from $F_t$,  $V_I$ be the vertex set spanned by the edges in $I$. Thus, $|I| \ge t$. If $z\in\{u',u''\}$, for convenience, assume that $z=u'$, then at most one edge from $I$ is incident to $u''$, which forces  $m(H) \ge t-1$.
Next suppose that $z\in V(H)$. If $\big|V(K_2) \cap V_I\big| \le 1$, then $m(H) \ge t-1$. If $V(K_2) \subseteq V_I$ and $u'u''\in I$, then $m(H) \ge t-1$. If $V(K_2) \subseteq V_I$ and $u'u'' \notin I$, then  there exist $u_1,u_2\in V(H)\setminus\{z\}$ such that $\{u_1u',u_2u''\}\subseteq I$,  thus $|(I\setminus\{u_1u',u_2u''\})\cup\{zu_1\}|\ge t-1$, i.e., $m(H)\ge t-1$. 
Next, we establish the necessity. Suppose that $|V(H)|\geq2t-1$ and
$v_1v_2,\ldots,v_{2t-3}v_{2t-2}$ form a matching in $H$. Thus, $V(H)\setminus\{v_1, \dots , v_{2t-2}\} \neq \varnothing$. Let $w \in V(H)\setminus\{v_1, \dots , v_{2t-2}\}. $
Then the triangles
$u'v_1v_2u',\ \ldots,\ u'v_{2t-3}v_{2t-2}u',\ u'u''wu'$
meet pairwise exactly at $u'$ and form a copy of $F_t$. This is a contradiction. 

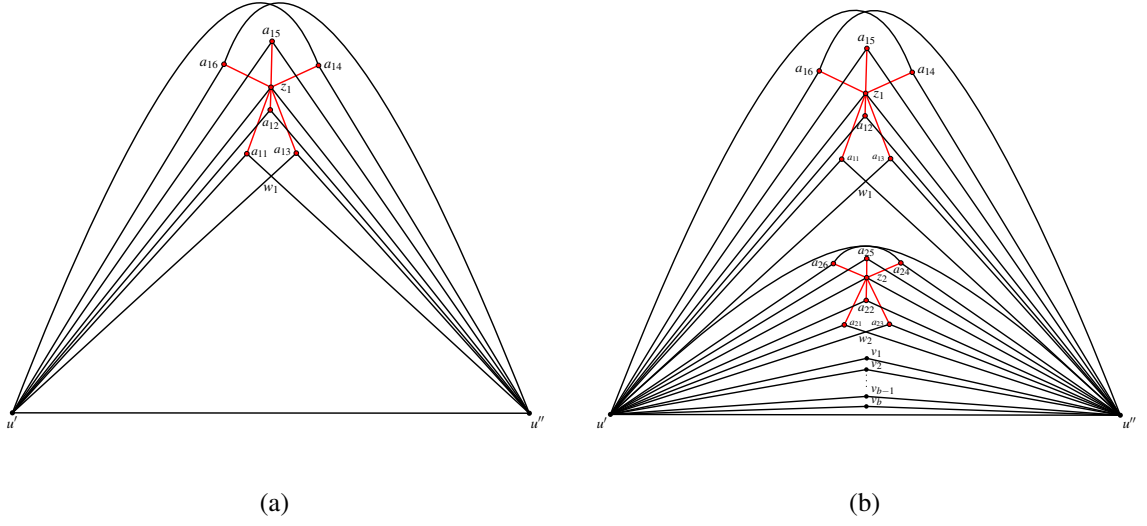
\begin{figure}[htbp]
	\centering
	
	%==================== (a) ====================
	\begin{minipage}[t]{0.48\textwidth}
		\centering
		\resizebox{\linewidth}{!}{%
			\begin{tikzpicture}[
				scale=1,
				transform shape,
				x=0.01cm,
				y=-0.01cm,
				blackedge/.style={
					draw=black,
					line width=1pt,
					line cap=round,
					line join=round
				},
				rededge/.style={
					draw=red,
					line width=1pt,
					line cap=round,
					line join=round
				},
				vertexlabel/.style={
					font=\fontsize{10}{10}\selectfont,
					text=black,
					inner sep=0pt,
					outer sep=0pt,
					fill=none,
					draw=none
				},
				crossinglabel/.style={
					font=\fontsize{10}{10}\selectfont,
					text=black,
					inner sep=0pt,
					outer sep=0pt,
					fill=none,
					draw=none
				}
				]
				
				%------------------------------------------------
				% Vertex positions traced from the original image.
				%------------------------------------------------
				\coordinate (uL)  at (25.00,1045.00);
				\coordinate (uR)  at (1284.00,1046.00);
				
				\coordinate (z1)  at (654.80,251.30);
				
				\coordinate (a11) at (653.00,306.50);
				\coordinate (a12) at (596.00,413.20);
				\coordinate (a13) at (716.46,411.85);
				
				\coordinate (a14) at (540.36,195.09);
				\coordinate (a15) at (657.46,139.15);
				\coordinate (a16) at (770.50,198.00);
				
				% Crossing point only
				\coordinate (w1) at (655.360062,467.797453);
				
				% Red star edges
				\foreach \p in {a11,a12,a13,a14,a15,a16}{
					\draw[rededge] (z1) -- (\p);
				}
				
				% Bottom edge
				\draw[blackedge] (uL) -- (uR);
				
				% Straight edges from u'
				\foreach \p in {z1,a11,a12,a13,a14,a15}{
					\draw[blackedge] (uL) -- (\p);
				}
				
				% Straight edges from u''
				\foreach \p in {z1,a11,a12,a13,a15,a16}{
					\draw[blackedge] (uR) -- (\p);
				}
				
				% Curved outer edges
				\draw[blackedge]
				(uL)
				.. controls (309.7106,305.8093)
				and (606.5631,-230.9078)
				.. (a16);
				
				\draw[blackedge]
				(uR)
				.. controls (1136.3792,678.1584)
				and (729.9053,-332.6859)
				.. (a14);
				
				% Actual graph vertices
				\foreach \p in {uL,uR}{
					\fill[black] (\p) circle[radius=1.8pt];
				}
				
				\foreach \p in {z1,a11,a12,a13,a14,a15,a16}{
					\filldraw[
					fill=red,
					draw=black,
					line width=0.35pt
					] (\p) circle[radius=1.6pt];
				}
				
				% Labels
				\node[vertexlabel, anchor=north]      at (25,1054)   {$u'$};
				\node[vertexlabel, anchor=north west] at (1288,1054) {$u''$};
				
				\node[vertexlabel, anchor=south] at (657,128) {$a_{15}$};
				\node[vertexlabel, anchor=east]  at (527,194) {$a_{16}$};
				\node[vertexlabel, anchor=west]  at (784,199) {$a_{14}$};
				
				\node[vertexlabel, anchor=west] at (679,254) {$z_1$};
				
				\node[vertexlabel, font=\fontsize{9}{9}\selectfont, anchor=north] at (655,333) {$a_{12}$};
				\node[vertexlabel, font=\fontsize{9}{9}\selectfont, anchor=west]  at (607,412) {$a_{11}$};
				\node[vertexlabel, font=\fontsize{9}{9}\selectfont, anchor=east]  at (705,407) {$a_{13}$};
				
				\node[crossinglabel, anchor=north, yshift=-6pt] at (w1) {$w_1$};
				
				% Bounding box
				\pgfresetboundingbox
				\path[use as bounding box] (0,0) rectangle (1327,1088);
				
			\end{tikzpicture}%
		}
		\vspace{0.3em}
		
		{\small (a)}
	\end{minipage}
	\hfill
	%==================== (b) ====================
	\begin{minipage}[t]{0.48\textwidth}
		\centering
		\resizebox{\linewidth}{!}{%
			\begin{tikzpicture}[
				scale=1,
				transform shape,
				x=0.01cm,
				y=-0.01cm,
				blackedge/.style={
					draw=black,
					line width=1pt,
					line cap=round,
					line join=round
				},
				rededge/.style={
					draw=red,
					line width=1pt,
					line cap=round,
					line join=round
				},
				vertexlabel/.style={
					font=\fontsize{10}{10}\selectfont,
					text=black,
					inner sep=0pt,
					outer sep=0pt,
					fill=none,
					draw=none
				},
				smalllabel/.style={
					vertexlabel,
					font=\fontsize{9}{9}\selectfont
				},
				crossinglabel/.style={
					vertexlabel
				}
				]
				
				%------------------------------------------------
				% Coordinates traced from the reference image.
				%------------------------------------------------
				
				% Common endpoints
				\coordinate (uL) at (45.000,1044.000);
				\coordinate (uR) at (1316.000,1046.000);
				
				% Upper star
				\coordinate (z1)  at (681.120,242.880);
				\coordinate (a11) at (621.500,407.000);
				\coordinate (a12) at (679.750,299.000);
				\coordinate (a13) at (743.330,405.670);
				\coordinate (a14) at (797.330,190.330);
				\coordinate (a15) at (684.330,130.670);
				\coordinate (a16) at (565.500,187.000);
				
				% Lower star
				\coordinate (z2)  at (683.750,702.750);
				\coordinate (a21) at (628.000,820.500);
				\coordinate (a22) at (682.500,759.500);
				\coordinate (a23) at (740.500,819.000);
				\coordinate (a24) at (769.000,666.000);
				\coordinate (a25) at (684.000,655.000);
				\coordinate (a26) at (600.800,667.600);
				
				% Bottom family
				\coordinate (v1)  at (684.000,904.000);
				\coordinate (v2)  at (683.000,932.000);
				\coordinate (vb1) at (683.000,999.000);
				\coordinate (vb)  at (683.000,1024.000);
				
				% Crossing labels only
				\coordinate (w1) at (681.490,462.196);
				\coordinate (w2) at (681.579,838.061);
				
				% Red star edges
				\foreach \p in {a11,a12,a13,a14,a15,a16}{
					\draw[rededge] (z1) -- (\p);
				}
				\foreach \p in {a21,a22,a23,a24,a25,a26}{
					\draw[rededge] (z2) -- (\p);
				}
				
				% Common bottom edge
				\draw[blackedge] (uL) -- (uR);
				
				% Upper copy
				\foreach \p in {z1,a11,a12,a13,a15,a16}{
					\draw[blackedge] (uL) -- (\p);
				}
				\foreach \p in {z1,a11,a12,a13,a14,a15}{
					\draw[blackedge] (uR) -- (\p);
				}
				
				\draw[blackedge]
				(uL)
				.. controls (447.9636,3.2742)
				and (675.2374,-132.0643)
				.. (a14);
				
				\draw[blackedge]
				(uR)
				.. controls (782.3014,-319.8481)
				and (617.2108,49.3222)
				.. (a16);
				
				% Lower copy
				\foreach \p in {z2,a21,a22,a23,a25,a26}{
					\draw[blackedge] (uL) -- (\p);
				}
				\foreach \p in {z2,a21,a22,a23,a24,a25}{
					\draw[blackedge] (uR) -- (\p);
				}
				
				\draw[blackedge]
				(uL)
				.. controls (141.8757,956.0478)
				and (226.3471,879.2479)
				.. (300.0000,821.2045)
				.. controls (600.6149,584.3002)
				and (721.0079,600.8900)
				.. (a24);
				
				\draw[blackedge]
				(uR)
				.. controls (1256.8335,992.1617)
				and (1162.0907,902.2478)
				.. (1060.0000,822.5096)
				.. controls (875.1987,678.1700)
				and (666.3206,555.7559)
				.. (a26);
				
				% Bottom fan
				\foreach \p in {v1,v2,vb1,vb}{
					\draw[blackedge] (uL) -- (\p);
					\draw[blackedge] (uR) -- (\p);
				}
				
				% Actual vertices
				\foreach \p in {uL,uR}{
					\fill[black] (\p) circle[radius=1.8pt];
				}
				\foreach \p in {v1,v2,vb1,vb}{
					\fill[black] (\p) circle[radius=1.6pt];
				}
				\foreach \p in {z1,a11,a12,a13,a14,a15,a16,z2,a21,a22,a23,a24,a25,a26}{
					\filldraw[
					fill=red,
					draw=black,
					line width=0.35pt
					] (\p) circle[radius=1.6pt];
				}
				
				% Labels
				\node[vertexlabel, anchor=north] at (24,1052) {$u'$};
				\node[vertexlabel, anchor=north] at (1339,1052) {$u''$};
				
				\node[vertexlabel, anchor=south] at (684,119) {$a_{15}$};
				\node[vertexlabel, anchor=east]  at (552,186) {$a_{16}$};
				\node[vertexlabel, anchor=west]  at (811,188) {$a_{14}$};
				\node[vertexlabel, anchor=west]  at (701,246) {$z_1$};
				
				\node[smalllabel, anchor=north] at (680,318) {$a_{12}$};
				\node[smalllabel, font=\fontsize{7}{7}\selectfont, anchor=west] at (636,405) {$a_{11}$};
				\node[smalllabel, font=\fontsize{7}{7}\selectfont, anchor=east] at (730,404) {$a_{13}$};
				
				\node[crossinglabel, anchor=north, yshift=-6pt] at (w1) {$w_1$};
				
				\node[smalllabel, anchor=south] at (684,649) {$a_{25}$};
				\node[smalllabel, anchor=east]  at (589,669) {$a_{26}$};
				\node[smalllabel, anchor=north] at (772,677) {$a_{24}$};
				\node[smalllabel, anchor=west]  at (710,704) {$z_2$};
				\node[smalllabel, anchor=north] at (684,771) {$a_{22}$};
				\node[smalllabel, font=\fontsize{7}{7}\selectfont, anchor=west] at (642,817) {$a_{21}$};
				\node[smalllabel, font=\fontsize{7}{7}\selectfont, anchor=east] at (728,817) {$a_{23}$};
				
				\node[smalllabel, anchor=north, yshift=-3pt] at (w2) {$w_2$};
				
				\node[smalllabel, anchor=south west] at (695,901) {$v_1$};
				\node[smalllabel, anchor=south west] at (695,930) {$v_2$};
				\node[smalllabel, anchor=south west] at (695,996) {$v_{b-1}$};
				\node[smalllabel, anchor=south west] at (695,1022) {$v_b$};
				
				% Vertical ellipsis
				\foreach \yy in {946,960,974}{
					\fill[black] (683,\yy) circle[radius=0.4pt];
				}
				
				% Bounding box
				\pgfresetboundingbox
				\path[use as bounding box] (0,0) rectangle (1358,1084);
				
			\end{tikzpicture}%
		}
		\vspace{0.3em}
		
		{\small (b)}
	\end{minipage}
		\par\vspace{1pt}
	\caption{$K_2\vee S_6$ and $K_2\vee (2S_6\cup bK_1)$}
	\label{Fig:1}
	
\end{figure}

\begin{remark}\label{rk222}

	If $a=0, b=0$, then $K_2$ is a $1$-planar graph. If $a=0, b\ge 1$, then $K_2\vee bK_1$ is a $1$-planar graph. Next, we only need to consider $a\ge1$ and $b\ge0$. Note that Figure \ref{Fig:1}$(a)$ exhibits a $1$-planar drawing of the graph $K_2 \vee  S_6$.  We embed another copy of $S_6$ inside the triangular region bounded by $u'w_1u''u'$, and connect every vertex to $u'$ and $u''$. This yields a $1$-planar drawing of $K_2\vee 2S_6$, as illustrated in Figure \ref{Fig:1}$(b)$. Repeating this construction, for the drawing of $K_2\vee iS_6$, we embed a new copy of $S_6$ inside the triangular face $u'w_iu''u'$ and connect all its vertices to $u'$ and $u''$. By carrying out this procedure $a$ times, we obtain a $1$-planar drawing of $K_2\vee aS_6$. Then we insert $b$ isolated vertices within the triangular region $u'w_au''u'$, and each of these vertices is joined to $u'$ and $u''$. This produces a $1$-planar drawing of $K_2\vee (aS_6\cup bK_1)$, and Figure \ref{Fig:1}$(b)$ shows the case $a=2$. Consequently, $K_2\vee (aS_6\cup bK_1)$ is 1-planar.

\end{remark}

For a face $f$, $\partial(f)$ denotes its boundary, and the size $d(f)$ is the number of the edges contained in its boundary with each repeated edge counted twice.
\begin{lemma}\label{lem:repair}
Let $G$ be a $1$-planar graph with $n\geq20$ vertices, $\rho=\rho(G)$, and 
$X=(x_v)_{v\in V(G)}^{\mathsf T}$
be a nonnegative Perron vector normalized by
$\max_{v\in V(G)}x_v=1$. Suppose that
$V(G)=\{u',u''\}\cup R\cup B,
R\subseteq N_{G}(u')\cap N_{G}(u''), |R|\geq\alpha n,$
where the vertex subsets are pairwise disjoint and $\alpha=\frac{4949}{5000}$. Let $G^*$ be obtained by retaining
$G[R\cup\{u',u''\}]$ and making each vertex of $B$ adjacent precisely
to $u',u''$. Then there exists $D'\subseteq E(G[R])$ such that
$|D'|\leq3$, $G^*-D'$ is $1$-planar, and
\[
X^T\bigl(A(G^*)-A(G^*-D')\bigr)X\leq\frac{38400}{\rho^2}.
\]
\end{lemma}

\noindent \emph{Proof.}
If $B=\varnothing$, then $D'=\varnothing$
and the statement holds trivally. Suppose that
$B\neq\varnothing$, and fix a good $1$-planar drawing of
$G[R\cup\{u',u''\}]$. 

\Needspace{5\baselineskip}
\begin{case}
There exist vertices $v_1,v_2\in V(G)\setminus\{u',u''\}$ such that edge $v_1u'$ crosses edge $v_2u''$ .
\end{case}

Let $w$ denote the cross point of $v_1u'$ and $v_2u''$ . Then $u'w$ and $u''w$ do not cross any other edges. Draw a curve connecting $u'$ and $u''$ closely along $u'wu''$  such that the region bounded by $u'u''$ and $u'wu''$ contains no vertices or edges. Then we embed $u'bu''$ without crossings inside this region for all $b\in B$.  Since no new crossing is created, hence $D'=\varnothing$
satisfies the conclusion.

\Needspace{5\baselineskip}
\begin{case}
For all $v_1,v_2\in V(G)\setminus\{u',u''\}$, edges $v_1u'$ and $v_2u''$ never cross each other.
\end{case}
Observe that the edges $\{u'r,u''r| r\in R\}$ form a plane drawing of  $K_{2,s}$ with $s=|R| \ge \lceil 20\alpha \rceil =20.$ Label $R=\{r_1,r_2, \dots , r_s\}$ and $r_{i+s}=r_i$ in counterclockwise order around $u'$, this yields a planar drawing of $K_{2,s}$. By Euler's formula, $K_{2,s}$ has exactly $s$ faces $Q_1,\dots,Q_s$, where $\partial(Q_i) = u'r_iu''r_{i+1}u'$.  If  $u'u'' \in E(G)$, since edge $u'u''$ is incident to every edge in $K_{2,s}$, the edge \(u'u''\) crosses no edge of this plane \(K_{2,s}\). Thus, $u'u''$ is contained in at most one quadrilateral region. Moreover, at least \(s-1\) of these faces avoid the interior of \(u'u''\). Hence there exists $ Q_i$ such that $Q_i$ does not contain $u'u''$. Note that $r_i$ is incident only with $Q_{i-1}$ and $Q_i$, and an edge
of $G[R]$ can cross at most one boundary edge of the plane $K_{2,s}$. Thus any edge in $E(R)$ crossing the quadrilateral region $Q_i$ must be incident to $r_i$ or $r_{i+1}$, otherwise, it crosses at least two boundary edges of $Q_i$. Therefore, only edges among $\{r_{i-1}r_{i+1}, r_ir_{i+1}, r_ir_{i+2}\}$ can cross the quadrilateral region $Q_i$. Let $L_0=\{v\in R:x_v> \frac{80}{\rho}\}$. If $L_0=\varnothing$, then $|L_0|=0 < \frac{n}{10}$. If $L_0\neq\varnothing$, by
Lemma~\ref{lem:edge} 
\[
80|L_0|<\sum_{v\in L_0}\rho x_v\leq2e(G)<8n.
\]
Consequently, $|L_0|< \frac{n}{10}$.
An ordering $4$-tuple $(r_{i-1},r_i,r_{i+1},r_{i+2})$ is called good 
if $\max \{x_{r_{i-1}},x_{r_i},$ $x_{r_{i+1}},x_{r_{i+2}}\} \le \frac{80}{\rho}$ and bad otherwise. Each vertex of $L_0$ belongs to four cyclic ordered $4$-tuples.
Hence the number of bad $4$-tuples is at most $4|L_0|$,  more than
$s-\frac{2n}{5}\geq\left(\alpha-\frac25\right)n>2$
good ordering $4$-tuples. Choose a good ordering $4$-tuples such that the quadrilateral region $Q_j$ bounded by $u'r_ju''r_{j+1}u'$ contains no $u'u''$. Let $D'$ be the set of edges that actually exist in $\{r_{j-1}r_{j+1}, r_jr_{j+1}, r_jr_{j+2}\}$. Then $|D'|\leq3$,  deleting $D'$ adding  paths $u'bu''$ inside $Q_j$ without crossings, where $b\in B$. This yields a $1$-planar drawing of $G^*-D'$, and
\[
X^T\bigl(A(G^*)-A(G^*-D')\bigr)X
\le 2\sum_{uv\in D'}x_{u}x_{v}
\leq2\cdot3\left(\frac{80}{\rho}\right)^2
=\frac{38400}{\rho^2}.
\]
This completes the proof of Lemma \ref{lem:repair}.

\noindent \emph{\textbf{Proof of Theorem \ref{thm1}}.}
Let $G\in\operatorname{SPEX}_{\mathcal{P}_1}(n,F_t)$ and 
$\rho=\rho(G)$, where $n \ge \max \{1.23 \times 10^{18}, 2t+1,  
\frac{2t}{\alpha},  $ $\frac{f(t)^2-f(t)+4}{2},$ $ \frac{g(t)^2-g(t)+4}{2}, 7t-12 \}$. Let $u'\in V(G)$ such that $x_{u'}=1$, $\pi=\{V_1,V_2,\dots,V_k\}$ be a partition of $V(G)$. The partition $\pi$ is called an equitable partition if every vertex in $V_i$ has the same number of neighbors in $V_j$ for all $i,j$. The quotient matrix $B=(b_{ij})$ associated with $\pi$ is a $k\times k$ matrix, where the $b_{ij}$ is the number of neighbors in $V_j$ of any vertex in $V_i$. The following claims hold.

\setcounter{claim}{0}
\begin{claim}\label{cl:connected}
$G$ is connected, $\rho\ge \sqrt{2n-4}$, and $x_v\ge \frac{1}{\rho}$ for every  $v\in V(G)$.
\end{claim}

\noindent \emph{Proof.}
Since $K_{2,n-2}$ is a $C_3$-free $1$-planar graph, $\rho\ge \rho(K_{2,n-2})=\sqrt{2n-4}$. Suppose for contradiction that $G$ is disconnected. Let $G_1,G_2$ be two connected components of $G$ such that $\rho=\rho(G_1)$. Take $u\in V(G_1),v\in V(G_2)$ such that $u,v$ lie on the boundaries of $G_1,G_2$, respectively. Then adding the edge $uv$ creates no new crossings. Construct a new graph $G^*=G +uv$. Since $uv$ is a cut edge, $uv$ does not belong to any triangle. So $G^*$ is an $F_t$-free $1$-planar graph. However,  $\rho(G^*)>\rho$, a contradiction. Therefore, $G$ is connected.
 Assume there exists $z \in V(G)$ with $x_z< \frac{1}{\rho}$. We get $\rho x_z=\sum_{w\in N_G(z)}x_w<1$, so $u'z\notin E(G)$.
Construct $G'$ from $G$ by deleting all edges incident to $z$ and adding the edge $u'z$. Note that $G'$ has a $1$-planar drawing by appropriately adjusting the position of the leaf vertex $z$. Hence $G'$ is still an $F_t$-free $1$-planar graph. And
\[
X^T\big(A(G')-A(G)\big)X
\ge2x_z\Big(1-\sum_{w\in N_G(z)}x_w\Big)
=2x_z(1-\rho x_z)>0.
\]
Thus, $\rho(G')>\rho$, a contradiction. Hence $x_z\ge \frac{1}{\rho}$ for every $z\in V(G)$.

\begin{claim}\label{cl:lower}
If $t\ge 2$, then $\rho \ge \frac{1+\sqrt{8n-15}}{2}$.
\end{claim}
\noindent \emph{Proof.}
Note that $K_2\vee (n-2)K_1$ is an $F_t$-free $1$-planar graph for $t\ge2$. We partition $V(K_2\vee (n-2)K_1)$ as $V(K_2)\cup V((n-2)K_1)$. Then the quotient matrix of $K_2\vee (n-2)K_1$ is $
	\begin{pmatrix}
		1 & n-2 \\
		2 & 0
	\end{pmatrix}.$
Therefore, $\rho \ge \frac{1+\sqrt{8n-15}}{2}$.

\vspace*{2mm}
Let $\eta=\frac{1}{14000}$ ,
$q_0=\sqrt{1-\frac{\eta}{2}}=\sqrt{\frac{27999}{28000}}$ and
$ N_c=\left\lceil\frac{32}{q_0^2\eta^4}\right\rceil
 < 1.23 \times 10^{18}$. Define 
 $L=\{v\in V(G):x_v>\eta\}$, $S=V(G)\setminus L$. From $n\ge N_c$, Lemma~\ref{lem:spectral} and Claim~\ref{cl:connected} yield $\rho^2\ge (2-\eta)n$, $q_0\sqrt{2n}\le \rho<\sqrt{2.01\,n}.$

\begin{claim}\label{cl:partition}

$|L|<\eta n$, $e(L)<4\eta n$, $e(S)\le 4n$, and
$e(L,S)\le(2+4\eta)n$.
\end{claim}
\noindent \emph{Proof.}
Since $\rho x_v\le d_G(v)$ for any $v\in L$, $d_G(v)>\eta q_0\sqrt{2n}$. Thus,
\[
\eta q_0\sqrt{2n}\cdot |L| < \sum_{v\in L}d_G(v)\le 2e(G)\le 8n,
\]
which implies $|L|<\frac{4\sqrt{2n}}{\eta q_0}$. Combined with $n\ge \frac{32}{q_0^2\eta^4}$, this simplifies to $|L|<\eta n$.
Note that both $G[L]$ and $G[S]$ are $1$-planar graphs. By Lemma~\ref{lem:edge}, $e(L)\le 4|L| \le 4\eta n$ and $e(S)\le 4|S|\le 4n$.
If $|L|\ge 2$ and $|S|\ge 2$, by Lemma~\ref{lem:bipartite}, then
$e(L,S)\le 2n+4|L|-12\le (2+4\eta)n.$
If $|L|\le 1$ or $|S| \le 1$, then trivially $e(L,S)\le n-1$, and the same upper bound holds.

\begin{claim}\label{cl:nonneighbors}
For every $u\in L$, $
\big|S\setminus N_G(u)\big|
\le2.01\left(\frac{1+11\eta+2\eta^2}{q_0}-q_0x_u\right)n.
$
In particular, $d_S(u')\ge(1-\xi_1)n$, where $\xi_1=\eta+2.01\left(\frac{1+11\eta+2\eta^2}{q_0}-q_0\right).$
\end{claim}
\noindent \emph{Proof.}
From the eigenvector equation
\[
\sum_{v\in N_G(u)\cap S}x_v
=\rho x_u-\sum_{v\in N_G(u)\cap L}x_v.
\]
By $\rho\ge q_0\sqrt{2n}$ and  Claim~\ref{cl:partition}, we have
\[
\begin{aligned}
\sum_{v\in N_G(u)\cap L}x_v
&\le \frac{1}{\sqrt{2n}{q_0}} \left( \sum_{v\in N_G(u)\cap L}\sum_{w\in N_G(v)\cap L}x_w + \sum_{v\in N_G(u)\cap L}\sum_{w\in N_G(v)\cap S}x_w \right)\\
&\le \frac{1}{\sqrt{2n}{q_0}}\left(2e(L)+\eta e(L,S)\right)\le \frac{(5\eta+2\eta^2)\sqrt{2n}}{q_0}.
\end{aligned}
\]
Thus,
\[
\sum_{v\in N_G(u)\cap S}x_v
\ge \Big(q_0 x_u-\frac{5\eta+2\eta^2}{q_0}\Big)\sqrt{2n}.
\]
On the other hand,
\[
\begin{aligned}
	\rho\sum_{v\in S} x_v &= \sum_{w\in L} d_S(w)x_w + \sum_{w\in S} d_S(w)x_w \\
	&\le e(L,S) + 2\eta e(S).
\end{aligned}
\]
Hence,
\[
\sum_{v\in S}x_v
\le \frac{e(L,S)+2\eta e(S)}{q_0\sqrt{2n}}
\le \frac{1+6\eta}{q_0}\sqrt{2n}.
\]
Recall from Claim \ref{cl:connected} that $x_v\ge \frac{1}{\rho}> \frac{1}{\sqrt{2.01n}}$ for each $v\in S$, so
\[
\big|S\setminus N_G(u)\big|\cdot \frac{1}{\sqrt{2.01n}}
\le \sum_{v\in S\setminus N_G(u)}x_v
\le \left(\frac{1+11\eta+2\eta^2}{q_0}-q_0 x_u\right)\sqrt{2n}.
\]
Thus, substitute $u=u',x_{u'}=1$, by $|S| \ge (1-\eta)n$, we have
\[d_S(u')=|S|-\big|S\setminus N_G(u')\big| \ge (1-\eta)n-2.01\left(\frac{1+11\eta+2\eta^2}{q_0}-q_0 x_{u'}\right)n=(1-\xi_1)n.\]

Define $\beta=\frac{1-19\eta-4\eta^2}{1+4\eta+\xi_1}\le1 $ and
$\xi=\xi_1+\xi_2$, where $\xi_2=\eta+2.01\left(\frac{1+11\eta+2\eta^2}{q_0}
-q_0\beta\right).$ Note $u' \in L$, $L \neq \varnothing. $

\begin{claim}\label{cl:second}
There must exist a vertex $u''\in L\setminus\{u'\}$ such that $x_{u''}\ge \beta>\frac{4983}{5000}$ and
$\big|N_G(u')\cap N_G(u'')\big|\ge (1-\xi)n.$\end{claim}
\noindent \emph{Proof.}
Since $x_{u'} = 1$, 
\[
\begin{aligned}
	\rho^2x_{u'}=\rho^2 &= \sum_{v\in N_G(u')}\sum_{w\in N_G(v)} x_w = \sum_{w\in V(G)} \big|N_G(w)\cap N_G(u')\big| x_w \le \sum_{w\in V(G)} d_G(w)x_w.
\end{aligned}
\]
By $x_w\le \eta$ for $w\in S$, we have
\[
\begin{aligned}
	\sum_{w\in V(G)} d_G(w)x_w
	&= \sum_{w\in L} d_S(w)x_w + \sum_{w\in L} d_L(w)x_w + \sum_{w\in S} d_G(w)x_w \\
	&\le \sum_{w\in L} d_S(w)x_w + 2e(L) + \eta\big(e(L,S)+2e(S)\big) \\
	&\le \sum_{w\in L} d_S(w)x_w + (18\eta + 4\eta^2)n.
\end{aligned}
\]
Combined with $\rho^2 \ge (2-\eta)n$,
\[
\sum_{w\in L} d_S(w)x_w \ge (2 - 19\eta - 4\eta^2)n.
\]
Thus,
\[
\sum_{w\in L\setminus\{u'\}} d_S(w)x_w > (1 - 19\eta - 4\eta^2)n > 0.
\]
Also,
\[
\sum_{w\in L\setminus\{u'\}} d_S(w)=e(L,S)-d_S(u') \le (1 + 4\eta + \xi_1)n.
\]
Therefore, there exists a vertex $u''\in L\setminus\{u'\}$ such that
\[
x_{u''} \ge \frac{1 - 19\eta - 4\eta^2}{1 + 4\eta + \xi_1} = \beta.
\]
By Claim \ref{cl:nonneighbors}, substitute $u=u''$,
\[d_S(u'')=|S|-\big|S\setminus N_G(u'')\big| \ge (1-\eta)n-2.01\left(\frac{1+11\eta+2\eta^2}{q_0}-q_0 x_{u''}\right)n\ge(1-\xi_2)n.\]
Thus,
\[
\big|N_G(u')\cap N_G(u'')\big|\ge d_G(u')+d_G(u'')-n\ge (1-\xi)n.
\]

\vspace*{2mm}
For convenience, let $R=N_G(u')\cap N_G(u'')$ and $B=V(G)\setminus(R\cup\{u',u''\})$.

\begin{claim}\label{cl:small}
 $x_v\le\frac{153}{2000}$ for all $v\in R\cup B$,
$\xi<1-\alpha$ and $|R|\ge\alpha n$.
\end{claim}        
\noindent \emph{Proof.}
% Editorial scope correction: the estimate is used for both R and B below.
For each $v\in R\cup B$, we first assert that $d_R(v)\le6$. Otherwise,
$G[N_R(v)\cup \{u',u'',v\}]$ contains $K_{3,7}$, contrary to
Lemma~\ref{lem:bipartite}.  Since $G[B]$ is $1$-planar graph,
\[
\sum_{v\in B}d_G(v)\le 7|B|+2e(B)\le 15|B|<15\xi n.
\]
From $\rho x_v \le d_G(v)$,
$\sum_{v\in B}x_v <\frac{15\xi n}{\rho}.$
 So $\rho x_v\le2+6+\sum_{w\in B}x_w$ for each $v\in R\cup B$. By $\rho^2\ge (2-\eta)n$ and $q_0\sqrt{2n}\le \rho<\sqrt{2.01\,n}$, we have
\[
x_v\le\frac{8}{\rho}+\frac{15\xi n}{\rho^2}
\le\eta^2+\frac{15\xi}{2-\eta}<\frac{153}{2000} \quad \text{and} \quad \xi<1-\alpha.
\]
 Consequently,
$|R|\ge(1-\xi)n>\alpha n$.

\begin{claim}\label{cl:empty}
$B=\varnothing$.
\end{claim}
\noindent \emph{Proof.}
Suppose for contradiction $B\neq\varnothing$ and set $s=|B|\ge 1$. By Lemma \ref{lm2}, there exists a sequence $v_1,v_2, \dots, v_s$ composed of vertices in $B$ such that $v_i$ has at most $7$ neighbors in $\{v_{i+1},\dots,v_s\}$.
Define $G_1=G$. For each $i$, iteratively construct $G_{i+1}$ from $G_i$ by deleting all edges incident to $v_i$ and adding  two new edges $u'v_i$, $u''v_i$. At step $i$, $\big|N_{G_i}(v_i)\cap\{u',u''\}\big|\le 1, \big|N_{G_i}(v_i)\cap B\big|\le 7, \big|N_{G_i}(v_i)\cap R\big|\le 6. $ By Claim~\ref{cl:small}, $x_w\le\frac{153}{2000}$ for all $w\in R\cup B$,  hence $
\sum_{w\in N_{G_i}(v_i)}x_w\le \frac{3989}{2000}.$
Let $c_0=\frac{21}{10000}$.
Recall $x_{u'}=1$ and $x_{u''}\ge\frac{4983}{5000}$. Then
\[
\begin{aligned}
\frac12 X^T\big(A(G_{i+1})-A(G_i)\big)X
&\ge x_{v_i}\Big(x_{u'}+x_{u''}
-\sum_{w\in N_{G_i}(v_i)}x_w\Big)\\
&\ge x_{v_i}\bigl(1+\frac{4983}{5000}-(1+\frac{1989}{2000})\bigr)\\
&=c_0x_{v_i}.
\end{aligned}
\]
Thus,
\[
\begin{aligned}
X^T\big(A(G_{s+1})-A(G)\big)X
\ge2c_0\sum_{v\in B}x_v
\ge\frac{2c_0|B|}{\rho}.
\end{aligned}
\]
We now use induction to prove that each $G_i$ is $F_t$-free. If $t=1$, then $G$ is $C_3$-free. Thus $u'u''\notin E(G)$. Adding edges $u'v_i,u''v_i$ cannot create a triangle, hence $G_i$ is $F_t$-free.
Let $t\ge 2$. Inductively assume $G_i$ is $F_t$-free but $G_{i+1}$ contains $F_t$. Thus, $E(F_t) \cap \{u'v_i,u''v_i\} \neq \varnothing$.  Let $z$ denote the common vertex of $F_t$.  Note that only $u'u''v_iu'$ contains $v_i$, hence $u'u''\in E(G_i)$ and $v_i\neq z$ . Moreover, $z=u'$ or $u''$. Without loss of generality, let $z=u'$. Since $|V(F_t)\setminus\{u',u''\}|=2t-1$ and $|R|\ge \alpha n\ge 2t$ for $n\ge \frac{2t}{\alpha}$, there exists a vertex $r\in R\setminus V(F_t)$.
Replace triangle $u'u''v_iu'$ by $u'u''ru'$, yielding an $F_t$ inside $G_i$, a contradiction. Therefore, $G_{s+1}$ is $F_t$-free. By Lemma~\ref{lem:repair}, there exists $D'\subseteq E(G[R])$ with $|D'|\le 3$ such that $G_{s+1}-D'$ is an $F_t$-free $1$-planar graph, and
\[
X^T\big(A(G_{s+1})-A(G_{s+1}-D')\big)X \le \frac{38400}{\rho^2}.
\]
Thus
\[
X^T\big(A(G_{s+1}-D')-A(G)\big)X
\ge\frac{2c_0}{\rho}-\frac{38400}{\rho^2}
=\frac{2}{\rho}\left(c_0-\frac{19200}{\rho}\right).
\]
By $\rho^2\ge (2-\eta)n, q_0\sqrt{2n}\le \rho<\sqrt{2.01\,n},$ and $n\ge N_c$, we have
$\rho\ge q_0\sqrt{2N_c}>\frac{19200}{c_0}$.
It means that $c_0-\frac{19200}{\rho}>0$, then $\rho(G_{s+1}-D')>\rho.$ This contradicts spectral maximality of $G$. Hence $B=\varnothing$.

The Claim~\ref{cl:empty} shows that $V(G)=\{u',u''\}\cup R$  and both $u',u''$ are adjacent to every vertex of $R$. Let $H=G[R]$, so $|V(H)|=|R|=n-2$.

\begin{claim}\label{cl:adjacent}
If $t=1$, then $G\cong K_{2,n-2}$ and $x_{u'}=x_{u''}=1$.
If $t\ge 2$, then $u'u''\in E(G)$, $x_{u'}=x_{u''}=1$, and $
\frac{2}{\rho}\le x_v \le \frac{2}{\rho-6}$ for all $v \in R.$
\end{claim}
\noindent \emph{Proof.}
By symmetry and  uniqueness of the
normalized Perron vector, $x_{u'}=x_{u''}=1$.
We split the proof into the following two cases.
\Needspace{5\baselineskip}
\begin{casee}
$t=1$.
\end{casee}
Note that $F_1\cong C_3$. If $u'u''\in E(G)$, then $G[\{u',u'',r\}]\cong C_3$. If $r'r''\in E(H)$, then $G[\{u',r',r''\}]\cong C_3$. Both cases lead to a contradiction. Thus $E(H)=\varnothing$ and $u'u''\notin E(G)$, i.e.,\ $G\cong K_{2,n-2}$.
\Needspace{5\baselineskip}
\begin{casee}
$t \ge 2$.
\end{casee}
 By Lemma~\ref{lem:corebounds} , $\Delta(H)\le6$. Let $x_z=\max_{v\in R}x_v$. Thus,
\[
\rho x_z = x_{u'}+x_{u''}+\sum_{w\in N_H(z)}x_w
\le 2+6x_z,
\]
so $x_z\le \frac{2}{\rho-6}$. For any $v\in R$, $\rho x_v \ge x_{u'}+x_{u''}=2$. Thus, $\frac{2}{\rho}\le x_v\le \frac{2}{\rho-6}$ for any $v\in R$.

Now suppose for contradiction $u'u''\notin E(G)$. If $m(H)\ge t$, then $H$ has at least $t$ independent edges, so $G[V(H)\cup\{u'\}]$ contains $F_t$, a contradiction. Thus, $m(H)\le t-1$. By
Lemma~\ref{lem:core} , $e(H)\le6(t-1)$. Construct $G''=K_2\vee (n-2)K_1$ from $G$ by deleting all edges of $H$ and adding edge $u'u''$. Obviously,  $G''$ is an $F_t$-free planar graph. Then
\[
\frac12 X^T\big(A(G'')-A(G)\big)X
\ge x_{u'}x_{u''}-\sum_{uv\in E(H)}x_u x_v
\ge\frac{4983}{5000}-\frac{24(t-1)}{(\rho-6)^2}.
\]
Let 
$f(t)=6+\sqrt{\dfrac{24(t-1)}{4983/5000}}$.
Since $n\ge \frac{f(t)^2-f(t)+4}{2}$, $\rho\ge \frac{1+\sqrt{8n-15}}{2}>f(t)$
by Claim~\ref{cl:lower}.
Thus, $\frac{4983}{5000}-\frac{24(t-1)}{(\rho-6)^2}>0$. So $\rho(G'')>\rho$. This is a contradiction.
Hence $u'u''\in E(G)$. 

The case $t=1$ has already been settled by Claim~\ref{cl:adjacent}.
In the remainder of the proof, assume that $t\ge2$.

\begin{claim}\label{cl:saturation}
$m(H)=t-2$ and $e(H)=6(t-2)$.
\end{claim}
\noindent \emph{Proof.}
Since $|V(H)|=n-2\ge 2t-1$, $m(H)\le t-2$ and $e(H)\le 6m(H)\le 6(t-2)$ by Lemmas~\ref{lem:core} and~\ref{lem:friendship}. The statement is trivial when $t=2$. Let $t\ge 3$ and suppose $e(H)\le 6(t-2)-1$. Choose a copy of $H^*$ on the vertex set $V(H)$, where $H^*=(t-2) S_6\cup (n-7t+12)K_1$. By Remark~\ref{rk222}, the graph $G^*=K_2\vee H^*$
is an $F_t$-free $1$-planar graph. Let
$a'=\big|E(H^*)\setminus E(H)\big|, b'=\big|E(H)\setminus E(H^*)\big|.$
Then $a'-b'=6(t-2)-e(H)\ge1$ and $0\le b'\le6(t-2)-1$.
Thus,
\[
\begin{aligned}
\rho(G^*)-\rho(G)&\ge\frac{8}{X^TX}
\left(\frac{a'}{\rho^2}-\frac{b'}{(\rho-6)^2}\right)\\
&\ge\frac{8}{X^TX}\left(\frac{b'+1}{\rho^2}-\frac{b'}{(\rho-6)^2}\right).
\end{aligned}
\]
The function $h(b')=\frac{b'+1}{\rho^2}-\frac{b'}{(\rho-6)^2}$ is decreasing
in $b'$ obviously. By $b'\le6(t-2)-1$ ,
$h(b')\ge\frac{6(t-2)}{\rho^2}-\frac{6(t-2)-1}{(\rho-6)^2}$.
Thus,
\[
\rho(G^*)-\rho(G)\ge
\frac{8\bigl(\rho^2-72(t-2)\rho+216(t-2)\bigr)}
{X^TX\rho^2(\rho-6)^2}.
\]
Define $
g(t)=36(t-2)+\sqrt{(t-2)(1296t-2376)}$. Since $n\ge \frac{g(t)^2-g(t)+4}{2}$, $\rho\ge \frac{1+\sqrt{8n-15}}{2}>g(t)$ by Claim~\ref{cl:lower}.
Thus, $\rho(G^*)>\rho(G)$, contradicting
extremality. Therefore, $e(H)=6(t-2)$. Combining this with
$e(H)\le6m(H)\le6(t-2)$, we obtain $m(H)=t-2$.
\par

From Claims~\ref{cl:adjacent} and \ref{cl:saturation}, $G=K_2\vee H$ with $e(H)=6m(H)$ and $m(H)=t-2$. By Lemma~\ref{lem:core}, $H$ has a bipartition $(A,D)$ with $|A|=m(H)=t-2$ and $d_H(a)=6$ for every $a\in A$. When $t=2$, $A=\varnothing$. This completes the proof of Theorem~\ref{thm1}.

\section{Proof of Theorem \ref{thm2}}\label{sec:thm2}

In this section, we use Theorem~\ref{thm1} to determine the extremal
graph. The graph construction and the matching number criterion have
already been proved in Lemma~\ref{lem:friendship} and Remark ~\ref{rk222}.
We first introduce two definitions. A good drawing is a drawing where edges intersect only at common endpoints or proper crossings, adjacent edges do not cross, any two independent edges cross at most once, and no three edges cross at a single interior point.
Given a drawing $\mathcal D$ of $G$, the planarization of $\mathcal D$ is the graph obtained by replacing each crossing point of $\mathcal D$ with a dummy vertex.  We then derive the drawing properties required to characterize the connected components of $H$.

\begin{lemma}\cite{Kleitman1976}.\label{lm:parity}
If $K_{3,3}$ has a good drawing in the plane, then the total number of
crossings is odd.
\end{lemma}

\begin{lemma}\label{lm:triples}
Let $K_{3,6}$ have bipartition $(U,W)$ with $|U|=3$ and $|W|=6$, and $C(v,w)$ be the number of crossings
between edges incident with $v$ and edges incident with $w$.
Let $\mathcal D$ be  a good $1$-planar drawing of $K_{3,6}$, then there exists a
partition $W=W_1\cup W_2$ with $|W_1|=|W_2|=3$ such that the following
statements hold
\begin{enumerate}[label=(\roman*)]
\item $C(v,w)=1$ for any $v,w \in W_i$  and $v\neq w(i=1,2)$.
\item $C(v,w)=0$ for any $v \in W_1$, $w \in W_2$.
\end{enumerate}

\end{lemma}

\noindent \emph{Proof.}
We first assert a good $1$-planar drawing of a graph $G$ on $n\geq3$
vertices has at most $n-2$ crossings. If $G$ has $c>0$ crossings,
we construct $G'$ from $G$ by retaining all crossing edges and deleting all other edges with $V(G')=V(G)$. Let $G^*$ be the planarization of $G'$. Then $G^*$ is a simple bipartite planar graph with bipartition $(S,T)$, where $S=V(G)$ and $T=V(G^*)\setminus V(G)$. Thus, $|V(G^*)|=n+c, e(G^*)=4c$ and $4c \le 2(n+c)-4$. Hence, $c \le n-2$.
The assertion is trivial when $c=0$. In particular, a good
$1$-planar drawing of $K_{3,4}$ has at most five crossings.

For convenience, let $W=\{v_0,v_1, \dots , v_5\}$. For any $v_i\in W$, since $v_i$ has three incident edges and each is crossed at most once, 
$
\sum_{v_j\in W\setminus\{v_i\}}C(v_i,v_j)\leq3.
$
For any three distinct vertices $v_i,v_j,v_k\in W$, the drawing on
$U\cup\{v_i,v_j,v_k\}$ is a good drawing of $K_{3,3}$. By
Lemma~\ref{lm:parity}, we have $
C(v_i,v_j)+C(v_i,v_k)+C(v_j,v_k)\equiv1\pmod2.$
Let $W_1=\{v_0\}\cup\{v_i\in W\setminus\{v_0\}|C(v_0,v_i)\text{ is odd}\}$,  $W_2=W\setminus W_1$. 

We next assert that $C(v_j,v_k)$ is odd if $v_j,v_k\in W_i$ for $i=1,2$, and even otherwise.  If $v_j$ or $v_k=v_0$, without loss of generality, let $j=0$, then $v_k\in W_1$ and $C(v_j,v_k)$ is odd. If $v_j,v_k\neq v_0$ and $v_j,v_k\in W_i$ for $i=1,2$, then $C(v_j,v_0)$ and $C(v_k,v_0)$ are both odd for $i=1$ and both even for $i=2$. From $C(v_j,v_0)+C(v_k,v_0)+C(v_j,v_k)\equiv 1\pmod{2}$, we conclude that $C(v_j,v_k)$ is odd.
For any $v_j\in W_1$, $v_k \in W_2$. If $j=0$, then $C(v_j,v_k)$ is even. If $v_j \neq v_0$, then $C(v_j,v_0)$ is odd and $C(v_k,v_0)$ is even. Also, $C(v_j,v_0)+C(v_k,v_0)+C(v_j,v_k)\equiv 1\pmod{2}$. Thus $C(v_j,v_k)$ is even. Therefore, for any $v_j \in W_1$, $v_k \in W_2$, we have $C(v_j,v_k)$ is even. If $|W_1|\ge 4$ or $|W_2|\ge 4$ , without loss of generality, let $|W_1| \ge 4$. Choose $W'=\{w_1,w_2,w_3,w_4\}\subseteq W_1$, the induced subgraph $K_{3,4}$ induced by $W'$ and $U$ has at least six crossings in the drawing $\mathcal D$. This contradicts the fact that $K_{3,4}$ has at most five crossings. Therefore, $|W_1|=|W_2|=3$.
Let $W_1= \{v_0,v_1,v_2\}, W_2=\{v_3,v_4,v_5\}$. Since $C(v_0,v_1)$, $C(v_0,v_2)$ are all odd and $C(v_0,v_1)+C(v_0,v_2)\le \sum_{y\in W\setminus \{v_0\} }C(v_0,y)\le 3$, we have $C(v_0,v_1)=C(v_0,v_2)=1$. Similarly, $C(v_1,v_2)=1$. Thus, statement $(i)$ follows. Since each $v\in W$ has at most three incident edges, $v$ has at most three crossing edges. Thus, we obtain $C(v_i,v_j)=0$ for $v_i\in W_1$ and $v_j\in W_2$. Therefore, statement $(ii)$ follows.

\begin{lemma}\label{lm:separation}
Let $G$ be a $1$-planar graph and $Q\subseteq G$ be a copy of
$K_{3,6}$ with bipartition $(Z,W)$, where $|Z|=3$.
If $b\in V(G)\setminus V(Q)$ has two distinct neighbors in $Z$,
then $N_G(b)\cap W=\varnothing$.
\end{lemma}

\noindent \emph{Proof.}
Fix a good $1$-planar drawing of $G$. Let $Z=\{z_1,z_2,z_3\}$, and choose distinct $z_1,z_2\in N_G(b)\cap Z$.
Suppose to the contrary that there exists a vertex $w \in W$ such that $bw \in E(G)$. By Lemma \ref{lm:triples}, let $\mathcal D$ be a $1$-planar drawing of $Q$, where $\mathcal D$ is the same as that in Lemma \ref{lm:triples}. Then $W=W_1\cup W_2$ with $|W_1|=|W_2|=3$. Let $
W_1=\{w_1,w_2,w_3\},
W_2=\{w_4,w_5,w_6\}.$ Without loss of generality, let $bw_1\in E(G)$.
For distinct vertices $w_j,w_k \in W_i$, $i\in\{1,2\}$, denote by $c_{jk}$ the
unique crossing between an edge incident with $w_j$ and an edge
incident with $w_k$. 

Note that there are exactly six crossings within the drawing of $Q$, involving twelve distinct edges. We construct $G'$ from $G$ by retaining $V(Q)\cup \{b\}$, the twelve crossing edges of $Q$  and $bz_1,bz_2,bw_1$, and deleting all other vertices and edges. Then we planarize $G'$ to obtain $G''$. Since each crossing edge crosses at most once, edges incident to $b$ have no crossings with other edges of $G''$. By the definition of a good $1$-planar drawing, $bz_1$, $bz_2$ and $bw_1$ do not cross each other. Therefore, $G''$ is a simple bipartite planar graph with bipartition $(S,T)$, where $S=V(Q)$, $T=V(G'')\setminus S$, and $|V(G'')|=16,e(G'')=27.$ Note that $G''$ contains two disjoint cycles $C^1=w_1c_{12}w_2c_{23}w_3c_{13}w_1,C^2=w_4c_{45}w_5c_{56}w_6c_{46}w_4$. By the Jordan Curve Theorem, $C^i$ partitions the plane into two disjoint regions $\Omega_i$ and $f_i$. By $(Z\cup W_{3-i})\cap(C^i )=\varnothing$ for $i\in \{1,2\}$ and Lemma \ref{lm:triples}, $Q[Z\cup W_{3-i}]$ is connected and disjoint from $C^i$ in the good $1$-planar drawing $\mathcal D$ of $Q$. Hence $Q[Z\cup W_{3-i}]$ belongs to one of the two regions separated by $C^i$. Without loss of generality, let $Q[Z\cup W_{3-i}]$ belong to $\Omega_i$. Thus, $Z$ and $C^{3-i}$ lie in $\Omega_i$. Since $b,z_1\notin V(C^i)$, $bz_1\in E(G'')$, and $z_1$ lies in $\Omega_i$, it follows that $b$ is also in $\Omega_i$. Similarly, all vertices in $V(G'')\setminus V(C^i)$ are contained in $\Omega_i$. Therefore, $f_i$ contains no vertices or edges, and $f_i$ is a face of $G''$ whose boundary is exactly $C^i$. For $i=1,2$, there exist two faces $f_1,f_2$ such that  $\partial (f_1)=C^1,\partial (f_2)=C^2$ and $d(f_1)=d(f_2)=6$. As $C^1\neq C^2$, $f_1\neq f_2$. Let $f$ be the total number of faces of $G''$. By Euler's formula, 
$
f\ge e(G'')-|V(G'')|+2=13
$.
As $G''$ is a simple bipartite graph, $d(f_i)$ is even and $d(f_i)\ge 4$, where $f_i$ are the faces of $G''$. By $\sum d(f_i)=2e(G'')$,
$
54=2e(G'')\ge \sum_{1\le i\le13}d(f_i)\ge 12+4\times(13-2)=56
$.
This is a contradiction. Therefore, $N_G(b)\cap W=\varnothing$.

\begin{lemma}\label{lm:starcomponent}
Let $H$ be a bipartite graph such that $K_2\vee H$ is a $1$-planar graph.
If $v\in V(H)$ and $d_H(v)=6$, then the connected component of $H$ containing
$v$ is $S_6$ with center vertex $v$.
\end{lemma}

\noindent \emph{Proof.}
Let $V(K_2)=\{u',u''\}$, $v\in V(H)$ such that $d_H(v)=6$,
$U=\{u',u'',v\}$ and $W=N_H(v)$.
Then $K_2\vee H$ contains $Q\cong K_{3,6}$.
Note that every $u\in V(H)\setminus(\{v\}\cup W)$ is adjacent to both $u'$ and
$u''$ in $K_2\vee H$. Lemma~\ref{lm:separation} shows that
$N_{H}(u)\cap W=\varnothing$.
Since $H$ is bipartite, $W$ is independent and
$H[\{v\}\cup W]\cong S_6$.

\noindent \emph{\textbf{Proof of Theorem \ref{thm2}}.}
If $t=1$, then $F_t=C_3$, and $K_{2,n-2}$ is a $C_3$-free $1$-planar graph.
Let $t\geq2$,  
define $H^*=(t-2)S_6\cup(n-7t+12)K_1.$ 
Hence $K_2\vee H^*$ is a $1$-planar graph by
Remark~\ref{rk222}.
Since $m(S_6)=1$, $m(H^*)=t-2.$
By $n\geq2t+1$ , we have
$|V(H^*)|=n-2\geq2t-1$. Lemma~\ref{lem:friendship} shows that
$K_2\vee H^*$ is $F_t$-free.
Let $G\in\operatorname{SPEX}_{\mathcal{P}_1}(n,F_t)$.
By Theorem~\ref{thm1}, there exist two distinguished vertices $u',u''\in V(G)$ such that
$R=N_G(u')\cap N_G(u'')=V(G)\setminus\{u',u''\}.$
Let $H=G[R]$. We distinguish two cases to discuss. If $t=1$, then Theorem \ref{thm1} $(i)$ immediately yields $G\cong K_{2,n-2}$. Assume $t\geq2$.
By Theorem~\ref{thm1}, $G=K_2\vee H$ and $H$ has a bipartition
$(A,D)$ such that
$|V(H)|=n-2$, $|A|=t-2$, and $d_H(v)=6$ for every $v\in A$.
If $A=\varnothing$, then $t=2$ and $e(H)=0$. Hence
$G\cong K_2\vee (n-2)K_1$.
Suppose that $A\neq\varnothing$. For each $v\in A$, let $H_v$ denote the connected component of $H$ containing $v$. Since $H$ is bipartite, $K_2\vee H$ is a $1$-planar graph and $d_H(v)=6$. By Lemma \ref{lm:starcomponent}, $H_v\cong S_6$, $V(H_v)=\{v\}\cup N_H(v)$. Moreover, as $v\in A$ and $(A,D)$ is the bipartition of $H$, $N_H(v)\subseteq D$, which implies $V(H_v)\cap A=\{v\}$. Therefore, each vertex in $A$ corresponds to a distinct component $S_6$ of $H$. Since $|A|=t-2$, $H$ contains exactly $t-2$ pairwise disjoint copies of $S_6$. By Theorem \ref{thm1}, $e(H)=6(t-2)$, so $H\cong (t-2)S_6\cup (n-7t+12)K_1$, that is, $G\cong K_2\vee ((t-2)S_6\cup (n-7t+12)K_1)$ for $t \ge 2$.

Thus $G$ is the unique extremal graph for every $t\geq1$.
This completes the proof of Theorem~\ref{thm2}.

\section{Acknowledgement}
AI-assisted tools were used to improve the language of this manuscript and to assist with checking calculations. The authors take full responsibility for the contents of the paper.

\end{document}